\documentclass[aap,preprint]{imsart}

\RequirePackage[OT1]{fontenc}
\RequirePackage{amsthm,amsmath,amssymb}
\RequirePackage[numbers]{natbib}
\RequirePackage[colorlinks,citecolor=blue,urlcolor=blue]{hyperref}
\RequirePackage{graphicx}
 
\startlocaldefs
\numberwithin{equation}{section}
\theoremstyle{plain}
\newtheorem{thm}{Theorem}[section]
\newtheorem{prop}{Proposition}[section]
\newtheorem{rem}{Remark}[section]
\newtheorem{cor}{Corollary}[section]
\endlocaldefs

\begin{document}

\begin{frontmatter}
\title{Non-Conventional Limits of Random Sequences Related to 
Partitions of Integers}
\runtitle{Non-Conventional Limits of Random Sequences}

\begin{aug}
\author{\fnms{Jordan} \snm{Stoyanov}\thanksref{m1}
\ead[label=e1]{stoyanovj@gmail.com}}
\and
\author{\fnms{Christophe} \snm{Vignat}\thanksref{t2,t3}
\ead[label=e2]{cvignat@tulane.edu}}

\thankstext{t2}{Corresponding Author}

\runauthor{J. Stoyanov and Ch. Vignat}

\affiliation{Institute of Mathematics and Informatics, Bulgarian Academy of Sciences Sofia, Bulgaria. Formerly at Newcastle 
University, UK \thanksmark{m1}}

\affiliation{L.S.S., CentraleSupelec, Universit\'{e} Paris-Sud XI, Orsay  \thanksmark{t2}}
\affiliation{Department of Mathematics, Tulane University \thanksmark{t3}}

\address{Institute of Mathematics and Informatics, Bulgarian Academy of Sciences,\\ Str. Acad. G. Bonchev Block 8, 1113 Sofia, Bulgaria\\
L2S-CentraleSupelec
3, Rue Joliot-Curie
91192 Gif sur Yvette, France\\
Tulane University, 6823 St. Charles Avenue, 
New Orleans, LA 70118, USA\\
\printead{e1}\\
\printead{e2}}

\end{aug}

\begin{abstract}
We deal with a sequence of integer-valued random variables $\{Z_N\}_{N=1}^{\infty}$ which is 
related to restricted partitions of positive integers. We observe that $Z_N=X_1+ \ldots + X_N$ 
for independent and bounded random variables $X_j$'s, so $Z_N$ has finite mean ${\bf E}Z_N$ and variance 
${\bf Var}Z_N$. We want to find the limit distribution of 
${\hat Z}_N=\left(Z_N-{\bf E}Z_N\right)/{\sqrt{{\bf Var}Z_N}}$ as $N \to \infty.$ 
While in many cases the limit distribution is normal,  
the main results established in this paper are  
that ${\hat Z}_N  \overset{d}{\to}  Z_{*},$ where 
$Z_{*}$ is a bounded random variable. We find explicitly the range of values of $Z_*$ and derive some   
properties of its distribution.  
The main tools used are moment generating functions, cumulant generating functions, moments and cumulants
of the random variables involved. Useful related topics are also discussed. 
\end{abstract}



\end{frontmatter}

\section{Introduction}
For any positive integer number $N$ we consider the following three objects: 

{\it Sample space} \ $\Omega_N = \{\omega\}$, \ $\omega=(k_1,k_2,\ldots,k_N),$ \ $0\leq k_j \leq 2^j - 1, j=1,\ldots, N.$ 
Think of $\omega$ as the outcome of an experiment, thus the number of all outcomes is $|\Omega_N| = 2^{N(N+1)/2}.$ To each 
$\omega$ we assign the same probability: \ ${\bf P}(\omega)=1/2^{N(N+1)/2}$ (uniform discrete distribution).  
 
{\it Sequence of integer numbers} \  $\{\alpha_{k}^{\left(N\right)}, \ k=0,1,\ldots,2^{N+1}-N-2\},$  defined 
by the equality 
\begin{equation}
\sum_{k=0}^{2^{N+1}-N-2}\alpha_{k}^{\left(N\right)}x^{k}=\prod_{i=0}^{N-1}\left(1+x^{2^{i}}\right)^{N-i}.\label{eq:gf}
\end{equation}

{\it Sequence of random variables} \ $\{Z_N, N=1,2,\ldots \},$ where 
\[
Z_N(\omega) = k_1+k_2+ \ldots + k_N \ \mbox{ for } \ \omega =(k_1,k_2,\ldots,k_N) \in \Omega_N. 
\]  

Now we are ready to involve together the above three objects when answering the question  
\[
{\bf P}[Z_N(\omega)=k] = ? 
\]
The range of values of $Z_N$ is the set $\{0, 1, \ldots, k, \ldots, \ldots, 2^{N+1}-N-2\}.$ The minimal value, $0$, is 
attained if and only if $k_1=k_2= \ldots=k_N = 0$, while the maximal value, $2^{N+1}-N-2,$ is possible if and only if  \ 
$k_j=2^j - 1, \ j=1,2,\ldots, N.$ In general we have the following relation: 
\begin{equation}
{\bf P}[Z_{N}=k] =\frac{\alpha^{(N)}_k}{2^{N(N+1)/2}}, \ k=0,1,\ldots,2^{N+1}-N-2.\label{eq:Zn}
\end{equation}

It is worth to tell that the 
 sequence \ $\{ \alpha_{k}^{\left(N\right)}, \ k=0,1, \ldots, 2^{N+1} - N - 2\}$ 
appears as entry A131823 in the Online Encyclopedia of Integer Sequences (OEIS) 
along with its generating function \eqref{eq:gf}. 

Notice that relation \eqref{eq:gf}, for $x=1,$ yields
\[
\sum_{k=0}^{2^{N+1}-N-2}\alpha_{k}^{\left(N\right)}=2^{N(N+1)/2},
\]
as it should be in order \eqref{eq:Zn} to define a proper probability distribution.

Another property of the sequence
$\left\{ \alpha_{k}^{\left(N\right)}\right\},$
\ not mentioned in OEIS, is its 
 combinatorial interpretation as a restricted partition function, namely:  
\[
\alpha_{k}^{\left(N\right)}=|\left\{ \left(k_{1},\dots,k_{N}\right):0\le k_{j} \le 2^j - 1, \ j=1, \ldots, N,  \ k_{1}+\cdots+k_{N}=k\right\}|,
\]
where \ $k=0, 1,2,\ldots, 2^{N+1}-N-2.$ To be clear, at the `end points' $k=0$ and $k=2^{N+1} -N -2$ we  have 
$\alpha_{k}^{\left(N\right)}=1.$ 
Moreover, the coefficients $\alpha_{k}^{\left(N\right)}$ satisfy the symmetry property $\alpha_{k}^{\left(N\right)}=\alpha_{k'}^{\left(N\right)}$ for any pair $\left(k,k'\right)$ of indices  such that $k+k'=2^{N+1}-N-2.$

For reader's convenience,  more details and properties of the sequence $\left\{ \alpha_{k}^{\left(N\right)}\right\}$ 
are summarized and given in Appendix at the end of the paper.

Now we focus our attention on the analysis of the random sequence $\{Z_N, \ N=1,2, \ldots \}.$ 
We will show further that the mean value and the variance of $Z_{N}$ are 
\[
\mu_{N}={\bf E}Z_{N}=2^{N}-\frac{N}{2}-1 \ \mbox{ and } \  
\sigma_{N}^{2}={\bf Var}Z_{N}=\frac{1}{9}\left(2^{2N}-\frac{3}{4}N-1\right), 
\]
and our specific goal is to study the standardized  
version of the random variable $Z_{N}:$ 
\[
\hat{Z}_{N}=\frac{Z_{N}-\mu_{N}}{\sigma_{N}}.
\]

A detailed analysis of $Z_{N}$ and $\hat{Z}_{N}$
is given in Sections 2 and 3. While usually the moments are the characteristics easier to 
calculate and use, here perhaps as a  surprise, it is easier to calculate explicitly the cumulants of  $\hat{Z}_{N}$, see 
 (\ref{eq:cumulants}). In Section 4 we prove that these cumulants have limits, and use an important  limit theorem, see Janson (1988), 
to conclude in Section 5 that there is a random variable,  $Z_*$, such that $ \hat{Z}_{N}  \stackrel{d}{\to}  Z_*$ as $N \to \infty.$ 
The limiting cumulants show that the limit distribution of  $Z_*$, say $\Lambda$, is not Gaussian. 
In fact, $\Lambda$ has a bounded support, the interval $[-3,3]$. This is the case $a=2$ of a more general 
model, Section 5, of a random sequence $\{Z_N(a)\}$ defined  for positive integer $a \geq 2.$ We find the exact bounded support for the 
limit distribution $\Lambda^{(a)}$ of the standardized sequence $\{\hat{Z}_N^{(a)}\}$ as $N \to \infty.$

The reader can consult the book by Su (2016) for other stochastic models and  
random sequences related to partitions of integers when the normal distribution \ ${\mathcal N}(0,1)$ appears as a 
limit distribution of quantities like our ${\hat Z}_{N}.$  Understandably, different models lead to different limit distributions, 
see, e.g., Morrison (1998) and Hwang and Zacharovas (2015).

\section{Properties of the random variables $Z_{N}$}

The random variable $Z_{N}$ as defined by \eqref{eq:Zn} can be characterized
as follows: Denote by $X_{i}$ a discrete random variable uniformly
distributed over the set $\left\{ 0,1,\ldots,2^{i}-1\right\}:$  
\[
{\bf P}[X_{i}=k] = \frac{1}{2^i}, \ k=0,1,\dots,2^{i}-1. 
\]

Consider $N$ independent random variables $X_{1},\dots,X_{N}$ and define
\[
\tilde{Z}_{N}=X_{1}+\cdots+X_{N}.\label{eq:Z_N sum X_i}
\]
Then we easily find that 
\[
{\bf P}[\tilde{Z}_{N}=k] = \frac{\alpha^{(N)}_k}{2^{N(N+1)/2}}, \ k=0,1,\dots,2^{N+1}-N-2.
\]
Hence, $Z_{N}$ and $\tilde{Z}_{N}$  take the same values with the same probabilities, i.e. 
\begin{equation}
\tilde{Z}_{N} \ \stackrel{d}{=} \ Z_{N}  \quad \Rightarrow \quad  Z_{N} \ \overset{d}{=} \ X_{1}+ \dots + X_{N}. \label{eq:2.1}
\end{equation}

This is a {\it stochastic representation} of $Z_{N}.$  
E.g., if $N=2,$ we have ${\bf P}[X_{1}=0] ={\bf P}[X_{1}=1] =\frac{1}{2},$
${\bf P}[X_{2}=0] ={\bf P}[X_{2}=1] ={\bf P}[X_{2}=2] ={\bf P}[X_{2}=3] =\frac{1}{4},$
 thus $Z_{2}=X_{1}+X_{2}$ takes values in the set $\{ 0,1,2,3,4\} $
with respective probabilities $\{ \frac{1}{8},\frac{1}{4},\frac{1}{4},\frac{1}{4},\frac{1}{8}\}.$ 

\begin{prop}
The moment generating function of $Z_{N}$ is 
\begin{align}
{\bf E}\text{e}^{tZ_{N}} & =\prod_{i=0}^{N-1}\left(\frac{1}{2}+\frac{1}{2}e^{t2^{i}}\right)^{N-i}.\label{eq:mgf1}
\end{align}
It has also another representation:  
\begin{equation}
{\bf E}\text{e}^{tZ_{N}}=\prod_{k=1}^{N}\frac{1}{2^{k}}\frac{1-e^{t2^{k}}}{1-e^{t}}.\label{eq:mgf2}
\end{equation}
\end{prop}

\begin{proof}
We use the moment generating function of the random variable $X_{i}$ and the independence of $X_1, \ldots, X_N$ to conclude that  
\[
{\bf E}\hbox{e}^{tX_{i}}=\frac{1}{2^{i}}\sum_{j=1}^{2^{i}-1}\hbox{e}^{tj}=\frac{1}{2^{i}}\frac{1-\hbox{e}^{t2^{i}}}{1-\hbox{e}^{t}}  \ 
\Rightarrow \ 
{\bf E}\hbox{e}^{tZ_{N}}=\prod_{i=1}^{N}\frac{1}{2^{i}}\frac{1-\hbox{e}^{t2^{i}}}{1-\hbox{e}^{t}}. 
\]
From the identity 
\[
\frac{1-\hbox{e}^{t2^{j}}}{1-\hbox{e}^{t}}=\prod_{i=0}^{j-1}\left(1+\hbox{e}^{t2^{i}}\right),
\]
we deduce that 
\begin{align*}
\prod_{j=1}^{N}\frac{1-\hbox{e}^{t2^{j}}}{1-\hbox{e}^{t}} =\prod_{j=1}^{N}\prod_{i=0}^{j-1}\left(1+\hbox{e}^{t2^{i}}\right)=\prod_{i=0}^{N-1}\prod_{j=i+1}^{N}\left(1+\hbox{e}^{t2^{i}}\right) 
 =\prod_{i=0}^{N-1}\left(1+\hbox{e}^{t2^{i}}\right)^{N-i}.
\end{align*}
Thus both (2.2) and (2.3) are established. 
\end{proof}
Based on the moment generating function \eqref{eq:mgf2}, we  derive
the following result. 

\begin{prop}
Besides \eqref{eq:2.1}, the random variable $Z_{N}$ admits
another stochastic representation: 
\begin{align}
Z_{N} & \overset{d}{=} V_{0}^{\left(0\right)}+\left(V_{0}^{\left(1\right)}+2V_{1}^{\left(1\right)}\right)+\left(V_{0}^{\left(2\right)}+2V_{1}^{\left(2\right)}+4V_{2}^{\left(2\right)}\right)+\cdots\label{eq:ZNV0V02V1}\\
 & +\left(V_{0}^{\left(N-1\right)}+2V_{1}^{\left(N-1\right)}+\cdots+2^{N-1}V_{N-1}^{\left(N-1\right)}\right).\nonumber 
\end{align}
Here all $V_{i}^{\left(j\right)}$ with  \ $0 \leq i \leq j \leq N-1,$ 
are independent Bernoulli random variables:  
\[
{\bf P}[V_{i}^{(j)}=0] ={\bf P}[V_{i}^{(j)}=1]=\frac{1}{2}.
\]
\end{prop}

\begin{proof}
The statement follows from the definition of the random variable $Z_{N}$ and the explicit expression of its 
moment generating function (\ref{eq:mgf2}). 
\end{proof}

\begin{rem}
The arguments used in the proof of Proposition 2 allow to state that a random variable $X_n,$ which is uniformly distributed over 
the set $\{0, 1, \ldots, 2^n - 1\},$ is a linear combination of $n$ independent 0--1 Bernoulli random variables 
with coefficients $2^0, 2^1, \ldots, 2^i, \ldots, 2^{n-1},$ respectively.   
\end{rem}

The stochastic representations (2.1) 
and \eqref{eq:ZNV0V02V1} allow to compute easily the first moments of $Z_{N}$. 

\begin{prop}
The mean value and the variance of the random variable $Z_{N}$ are  
\[
\mu_{N}={\bf E}Z_{N}=2^{N}-\frac{N}{2}-1 \ \mbox{ \ and \ } \ 
\sigma_{N}^{2}={\bf Var}Z_{N}=\frac{1}{9}\left(2^{2N}-\frac{3}{4}N-1\right).
\]
\end{prop}

\begin{proof}
First, from (2.1) 
we have 
\[
{\bf E}X_{k}=\frac{2^{k}-1}{2} \ \Rightarrow \ {\bf E}Z_{N}=\sum_{k=1}^{N}\frac{2^{k}-1}{2}=2^{N}-\frac{N}{2}-1.
\]
Second, in view of the independence of the random variables involved,
the variance of $Z_{N}$ is  
\[
\sigma_{N}^{2}=\sum_{k=1}^{N}{\bf Var}X_k =\sum_{k=1}^{N}\frac{2^{2k}-1}{12}=\frac{1}{9}\left(2^{2N}-\frac{3}{4}N-1\right).
\]
\end{proof}

\section{Properties of the random variable ${\hat Z}_N$}

In Proposition 2.3, we have found $\mu_n$ and $\sigma_N^2$, so the standardized random variable
\begin{equation}
\hat{Z}_{N}=\frac{Z_{N}-\mu_{N}}{\sigma_{N}}\label{eq:standard}
\end{equation} 
is explicitly defined, ${\bf E}{\hat Z}_N=0, {\bf Var}{\hat Z}_N =1,$ and we are looking for its limit as $N \to \infty.$ 
 Since for any $N$, \ $Z_N$ is the partial sum of the sequence of independent random variables 
$\{X_j\}$ having increasing but finite means and variances,  we may suggest that the 
asymptotic distribution of  $\hat{Z}_{N}$ is normal. This is equivalent to saying that the sequence of random variables 
$\{X_j\}$ satisfies the central limit theorem (CLT). 
It is well-known that all depends 
on the behaviour of the variance of $X_N$, the last term,  and its `contribution' to the variance of $Z_N.$ 
By using the specific structure of the random variables $X_j, j=1,\ldots, N$, we can easily show, e.g., 
that the classical Lindeberg's condition (see Lo\`eve (1977), Petrov (1995) or Shiryaev (2016)) 
is not satisfied. 
No conclusion, however, because this condition is only sufficient for the validity of the CLT. 
Interestingly, there are other conditions involved when studying the CLT, and they also fail to hold. These are   
 U.A.N. (uniform asymptotic negligibility) condition, and Feller's condition. Regarding all these conditions, Lindeberg's, U.A.N., 
Feller's, useful illustrations are given in Stoyanov (2013), Section 17. 

Our goal here is to find and/or characterize the limit distribution of $\hat{Z}_{N}$
as $N\to\infty$. For this we need to establish first the following result.

\begin{prop}
The moment generating function of $\hat{Z}_{N}$, see \eqref{eq:standard}, is 
\begin{equation}
\psi_N(t)={\bf E}\hbox{e}^{t\hat{Z}_{N}}\hspace{-0.1cm}=\hspace{-0.1cm}\prod_{k=1}^{N}\left(\frac{1}{2^{k}}\,\frac{\sinh\left(\frac{t}{2\sigma_{N}}\,2^{k}\right)}{\sinh\left(\frac{t}{2\sigma_{N}}\right)}\right)=\hspace{-0.2cm}\prod_{1\le l\le k\le N}\hspace{-0.2cm}\cosh\left(2^{k-l}\,\frac{t}{2\sigma_{N}}\right), \ t \in {\mathbb R}.
\label{eq:cf}
\end{equation}
Its cumulant generating function is 
\[
\log \psi_N(t)=\sum_{k=1}^{N}\log\left(\frac{1}{2^{k}}\,\frac{\sinh\left(\frac{t}{2\sigma_{N}}\,2^{k}\right)}{\sinh\left(\frac{t}{2\sigma_{N}}\right)}\right).
\]
Moreover, denoting by $\kappa_j^{(N)}=\kappa_j({\hat Z}_N)$ the $j$th order cumulant of $\hat{Z}_{N}, \ j=1,2, \ldots,$ we have  
\begin{equation}
  \kappa_{2n-1}^{\left(N\right)}=0,\label{eq:oddcumulants}
\end{equation}
\begin{equation}
  \kappa_{2n}^{\left(N\right)}=\frac{9^{n}}{4^{n}-1}\,\frac{4^{n}\left(4^{nN}-N-1\right)+N}{\left(4^{N}-\frac{3}{4}N-1\right)^{n}}\,\frac{B_{2n}}{2n}.\label{eq:cumulants}
\end{equation}
Here $B_{n}, n=1,2,\ldots,$ are the  Bernoulli numbers. Recall that $\{B_{n}\}_{n=0}^{\infty}$
are defined as the coefficients of the generating function 
\[
\sum_{n=0}^{\infty}\frac{B_n\,t^n}{n!}=\frac{t}{e^t-1}.
\]
\end{prop}
\begin{proof}
The first step is to use an elementary identity:  
\[
\frac{1}{2^{k}}\,\frac{\sin\left(2^{k}x\right)}{\sin x}=\prod_{l=1}^{k}\cos\left(2^{k-l}x\right) \ \Rightarrow \ 
{\bf E}\hbox{e}^{t\hat{Z}_{N}}=\prod_{1\le l\le k\le N}\cosh\left(2^{k-l}\,\frac{t}{2\sigma_{N}}\right).
\]
The next is to write the cumulant generating function of $\hat{Z}_{N}$
in terms of the cumulants. We have 
\[
\log \psi_N(t)=\sum_{k=1}^{N}\log\left(\frac{1}{2^{k}}\,\frac{\sinh\left(\frac{t}{2\sigma_{N}}\,2^{k}\right)}{\sinh\left(\frac{t}{2\sigma_{N}}\right)}\right)=\sum_{n=1}^{\infty}\kappa_{2n}^{\left(N\right)}\,\frac{t^{2n}}{2n!}.
\]
It follows that $\kappa_{2n-1}^{\left(N\right)}=0,$ as claimed in
\eqref{eq:oddcumulants}.

From the identity \cite[1.518.1]{Gradshteyn}, for any real  $c$ and $t$ such that  $0<ct<\pi,$
\[
\log\sinh\left(ct\right)=\log\left(ct\right)+\sum_{n=1}^{\infty}\frac{B_{2n}}{2n}\left(2c\right)^{2n}\,\frac{t^{2n}}{2n!},
\]
we obtain 
\begin{align*}
\log \psi_{N}(t) & =\sum_{k=1}^{N}\left(\log\frac{1}{2^{k}}+\log\sinh\left(\frac{t}{2\sigma_{N}}\,2^{k}\right)-\log\sinh\left(\frac{t}{2\sigma_{N}}\right)\right)\\
 & =\sum_{k=1}^{N}\left(\log\frac{1}{2^{k}}+\log\left(\frac{t}{2\sigma_{N}}\,2^{k}\right)-\log\left(\frac{t}{2\sigma_{N}}\right)\right)\\
 & +\sum_{k=1}^{N}\sum_{n=1}^{\infty}\frac{B_{2n}}{2n}\,\frac{t^{2n}}{2n!}\left\{ \frac{2^{2kn}}{\sigma_{N}^{2n}}-\frac{1}{\sigma_{N}^{2n}}\right\} .
\end{align*}
Therefore, 
\begin{align*}
\log \psi_N(t) & =\sum_{k=1}^{N}\sum_{n=1}^{\infty}\frac{B_{2n}}{2n}\,\frac{2^{2nk}-1}{\sigma_{N}^{2n}}\,\frac{t^{2n}}{2n!}\\
 & =\sum_{n=1}^{\infty}\frac{B_{2n}}{2n}\left(\sum_{k=1}^{N}\frac{2^{2nk}-1}{\sigma_{N}^{2n}}\right)\frac{t^{2n}}{2n!}\\
 & =\sum_{n=1}^{\infty}\frac{B_{2n}}{2n}\,\frac{4^{n}\left(4^{nN}-N-1\right)+N}{\sigma_{N}^{2n}\left(4^{n}-1\right)}\,\frac{t^{2n}}{2n!}.
\end{align*}
By using the expression for $\sigma_{N}^{2},$ we find exactly 
the value \eqref{eq:cumulants} for the even order cumulants $\kappa_{2n}^{\left(N\right)}.$
Hence, the desired result. 
\end{proof}
Knowing explicitly the cumulants $\kappa_{2n-1}^{\left(N\right)}$
and $\kappa_{2n}^{\left(N\right)}$ allows in turn to write another
stochastic representation for the random variable $\hat{Z}_{N}.$ Let us introduce first some notation. 
Here and below the standard abbreviation `i.i.d.' is used for `independent and identically distributed' random variables. 
Define the following random variable: 
\begin{equation}
X=\sum_{k=1}^{\infty}\frac{U_{k}}{2^{k}},\label{eq:X}
\end{equation}
where $\left\{ U_{k}\right\} _{k=1}^{\infty}$ are i.i.d. random variables
with continuous uniform distribution over the interval $\left[-3,3\right],$
and consider an infinite sequence  $X_{0,}X_{1}, X_2, \ldots$
 of independent copies of $X.$ 

Let now $L$ be a random variable with the square
hyperbolic secant distribution, its density is 
\begin{equation}
\label{sech}
f_{L}\left(x\right)=\frac{2\pi}{\left(\hbox{e}^{\pi x}+\hbox{e}^{-\pi x}\right)^{2}}, \ x\in\mathbb{R}.
\end{equation}
It is easy to find the cumulants of $L,$ namely: 
\[
\kappa_{2n+1}\left(L\right)=0,\,\,\,\,\kappa_{2n}\left(L\right)=-6^{2n}\frac{B_{2n}}{2n}.
\]

The next step is to define a complex-valued random variable $W,$ as follows:
\[
W=6{\bf i} L, \ \mbox{ where } \ {\bf i} = {\sqrt {-1}}, 
\]
and let $W_{1},W_{2},\ldots,$ be i.i.d. copies of $W.$ Based on these, we introduce the random variable 
\[
Y=\sum_{k=1}^{\infty}\frac{W_{k}}{2^{k}}
\]
and consider an infinite sequence  $Y_{1},Y_{2},\ldots$ of independent copies of \ $Y$. 

\begin{thm}
With the notation 
\[
c_{N}=\sqrt{4^{N}-\frac{3}{4}N-1},
\]
the random variable $\hat{Z}_{N}$ has the following stochastic representation:
\begin{equation}
\hat{Z}_{N} \ \overset{d}{=} \ \frac{1}{c_{N}}\left[2^{N}X_{0}+\left(Y_{1}+\dots+Y_{N+1}\right)+\frac{1}{2}\left(X_{1}+\dots+X_{N}\right)\right].\label{eq:ZN stochastic}
\end{equation}
\end{thm}

\begin{proof}
The idea is to use the properties of the cumulants of a random variable and the one-to-one 
relations between the moments and the cumulants; see Shiryaev (2016). 
Recall that for any random variable $\xi,$ \ $\kappa_{n}\left(\lambda\xi\right)=\lambda^{n}\kappa_{n}\left(\xi\right)$
for $\lambda\in\mathbb{R},$ and if $\eta$ is another random variable
independent of $\xi,$ we have \ $\kappa_{n}\left(\xi+\eta\right)=\kappa_{n}\left(\xi\right)+\kappa_{n}\left(\eta\right).$ 
The latter property justifies the use in the literature of the term `semiinvariants' instead of cumulants; see Janson (1988) 
or Shiryaev (2016). 
Moreover, if $\xi$ is symmetric with all moments finite, then all odd order cumulants are zero:
$\kappa_{2n-1}\left(\xi\right)=0.$ If $\xi$ and its distribution
are uniquely determined by the moment sequence $\left\{ {\bf E}\left[\xi^{k}\right],\thinspace\thinspace k=1,2,\dots\right\},$
and hence, by the cumulant sequence $\left\{ \kappa_{n}\left(\xi\right),n=1,2,\dots\right\} ,$
in this case a converse statement is valid:  If all cumulants
$\kappa_{n}\left(\xi\right),n=1,2,\dots$ are finite and $\kappa_{2n-1}\left(\xi\right)=0,n=1,2,\dots,$
then $\xi$ is symmetric. Also, assuming that if 
$\xi$ and $\eta$ are independent random variables, each uniquely determined by its moments, 
hence also by its cumulants, and  if $\zeta$ is a random variable such that for
all $n,$ \ $\kappa_{n}\left(\zeta\right)=\kappa_{n}\left(\xi\right)+\kappa_{n}\left(\eta\right),$
one can conclude that $\zeta\overset{d}{=}\xi+\eta.$

We have found, see \eqref{eq:cumulants}, the cumulants $\kappa_{2n}^{\left(N\right)}\left(=\kappa_{2n}\left(\hat{Z}_{N}\right)\right),$
and we want to show that
\begin{equation}
\kappa_{2n}\left(\hat{Z}_{N}\right)=\kappa_{2n}\left(X_{0}^{\left(N\right)}\right)+\kappa_{2n}\left(Y^{\left(N\right)}\right)+\kappa_{2n}\left(T^{\left(N\right)}\right),\label{eq:kappazn}
\end{equation}
where
\[
X_{0}^{\left(N\right)}=\frac{2^{N}}{c_{N}}X_{0}, \ Y^{\left(N\right)}=\frac{1}{c_{N}}\left(Y_{1}+\dots+Y_{N+1}\right), 
 \ T^{\left(N\right)}=\frac{1}{2c_{N}}\left(X_{1}+\dots+X_{N}\right).
\]
Since the random variables $X_{0},X_{1},X_{2},\dots, \ Y_{1}, Y_{2},\dots$ have simple structure,
they are all independent and symmetric with easily computable cumulants $\kappa_{2n}\left(X_{i}\right),\kappa_{2n}\left(Y_{i}\right),$
 we write explicitly the cumulants $\kappa_{2n}\left(X_{0}^{\left(N\right)}\right),\kappa_{2n}\left(Y_{0}^{\left(N\right)}\right)$
and $\kappa_{2n}\left(T^{\left(N\right)}\right).$ Then we find explicitly the cumulant $\kappa_{2n}\left(\hat{Z}_{N}\right):$   
\begin{align*}
\kappa_{2n}\left(\hat{Z}_{N}\right)=\kappa_{2n}^{\left(N\right)} & =\frac{1}{c_{N}^{2n}}\,6^{2n}\,\frac{B_{2n}}{2n}\,\frac{\left(2^{2nN}-N-1\right)+N\,2^{-2n}}{2^{2n}-1}\\
 & =\frac{1}{c_{N}^{2n}}\,6^{2n}\,\frac{B_{2n}}{2n}\sum_{k = 1}^{\infty}2^{-2kn}\left(2^{2nN}-N-1+N\,2^{-2n}\right)\\
 & =\frac{1}{c_{N}^{2n}}\,6^{2n}\,\frac{B_{2n}}{2n}\sum_{k = 1}^{\infty}\left(2^{2\left(N-k\right)n}-\left(N+1\right)2^{-2kn}+N\,2^{-2n\left(k+1\right)}\right).
\end{align*}

The last sum can be written as a sum of three terms, by keeping the factor before it, of course.  
The first term can be expressed as a cumulant, namely: 
\[
\frac{1}{c_{N}^{2n}}\,6^{2n}\,\frac{B_{2n}}{2n}\sum_{k = 1}^{\infty}2^{2\left(N-k\right)n}=\kappa_{2n}\left(X_{0}^{\left(N\right)}\right).
\]
The second term is 
\[
-\left(N+1\right)\frac{6^{2n}}{c_{N}^{2n}}\,\frac{B_{2n}}{2n}\sum_{k = 1}^{\infty}2^{-2kn}=\kappa_{2n}\left(Y_{0}^{\left(N\right)}\right).
\]
Finally, the third term is 
\[
\frac{N}{c_{N}^{2n}}\,6^{2n}\,\frac{B_{2n}}{2n}\sum_{k=1}^{\infty}2^{-2n\left(k+1\right)}=\kappa_{2n}\left(T^{\left(N\right)}\right).
\]
These three observations and the above comments on  properties of the cumulants justify
the equality \eqref{eq:kappazn}. Since 
\[
\kappa_{2n-1}\left(X_{0}^{\left(N\right)}\right)=0, \ \kappa_{2n-1}\left(Y^{\left(N\right)}\right)=0, \ 
\kappa_{2n-1}\left(T^{\left(N\right)}\right)=0 
\]
\[
 \Rightarrow \quad \kappa_{2n-1}\left(\hat{Z}_{N}\right)=0,
\]
the proof of the representation \eqref{eq:ZN stochastic} is completed.
\end{proof}

\begin{rem}
By noting that
\[
\lim_{N\to\infty}\frac{2^{N}}{c_{N}}=1,
\]
we see that, as $N\to\infty,$ only the first term in \eqref{eq:ZN stochastic}
remains. This follows from the fact that   
\[
Y^{\left(N\right)}=\frac{1}{c_{N}}\left(Y_{1}+\dots+Y_{N+1}\right)  \overset{P}{\to}  0  \ 
\mbox{ and } \ T^{\left(N\right)}=\frac{1}{c_{N}}\left(X_{1}+\dots+X_{N}\right)  \overset{P}{\to}  0. 
\]
For these two relations, we just apply Chebyshev inequality. Therefore,  since $X_0^{(N)}, Y^{(N)}$ and $T^{(N)}$ are independent, 
we refer to Slutsky's theorem and  conclude that 
\[
\hat{Z}_{N} \ \overset{d}{\to} \ X_{0} \ \mbox{ as } \ N \to \infty.  
\]
\end{rem}

\section{Asymptotic Results}

The limiting values of the cumulants $\kappa_{j}^{\left(N\right)}=\kappa_{j}\left(\hat{Z}_{N}\right)$
as $N\to\infty$ can be explicitly obtained from the expressions \eqref{eq:oddcumulants} and  \eqref{eq:cumulants}. 
\begin{cor} 
There is a random variable, $Z_*$, such that 
 $\hat{Z}_{N}\overset{d}{\to}Z_{*}$ as $N\to\infty$ and if $\Lambda$ is the distribution function of $Z_*$, then 
the cumulants of $\Lambda,$ $\kappa^*_n = \kappa_n (Z_*), n=1,2, \ldots, $ are as follows: 
\[
\kappa_{2n-1}^{*}=\lim_{N\to\infty}\kappa_{2n-1}^{\left(N\right)}=0, \quad \kappa_{2n}^{*}=\lim_{N\to\infty}\kappa_{2n}^{(N)}
=\frac{B_{2n}}{2n}\,\frac{6^{2n}}{2^{2n}-1},\,\,n=1,2,\ldots
\]
This clearly shows that the limit distribution is not Gaussian. 
\end{cor}

\begin{rem}
The interesting question now is about the behaviour of the cumulants $\kappa_{2n}^{*}$
for large $n.$ It is clear that  
\[
\kappa_{2n}^{*} \ \approx \ \frac{B_{2n}}{2n}\,3^{2n} \ \mbox{ for large } \ n.
\]
Notice however that \ $\frac{B_{2n}}{2n}\,3^{2n}$ \ is exactly the cumulant
of order $2n$ of the uniform absolutely continuous distribution on
the interval $\left[-3,3\right].$ This fact can eventually be used to find approximations for 
the limit distribution $\Lambda.$
\end{rem}

Another conclusion can be derived from Corollary 8.

\begin{prop}
(a) The limiting random variable $Z_*$ and its distribution
$\Lambda$ are uniquely determined by both the moments and the cumulants.

(b) Moreover, $Z_*$ admits the following stochastic representation:
\begin{equation}
Z_* \ \overset{d}{=} \ \sum_{i=1}^{\infty}\frac{U_{i}}{2^{i}},\label{eq:stochastic}
\end{equation}
where, as before, $U_{1},U_{2},\ldots,$ are independent absolutely continuous random variables which are uniformly distributed on
$\left[-3,3\right].$ 

As a consequence, the variable $Z_*$ is bounded and its 
 support is \ $\left[-3, 3\right]$. 
\end{prop}

\begin{proof}
We easily find simple expressions for the cumulants of $U_{i}$, $i= 1, 2, \ldots$, namely: 
\[
\kappa_{2n-1}\left(U_{i}\right)=0, \ \kappa_{2n}\left(U_{i}\right)=6^{2n}\frac{B_{2n}}{2n}. 
\]
Let us rewrite the cumulant of order $2n$ of $Z_*$ as follows:
\begin{align*}
\kappa_{2n}^{*} & =\frac{B_{2n}}{2n}\,\frac{6^{2n}}{2^{2n}-1}=6^{2n}\,\frac{B_{2n}}{2n}\,\frac{2^{-2n}}{1-2^{-2n}}=\kappa_{2n}\left(U\right)\sum_{i=1}^{\infty}\left(\frac{1}{2^{i}}\right)^{2n}\\
 & =\sum_{i=1}^{\infty}\frac{\kappa_{2n}\left(U_{i}\right)}{\left(2^{i}\right)^{2n}}=\sum_{i=1}^{\infty}\kappa_{2n}\left(\frac{U_{i}}{2^{i}}\right)=\kappa_{2n}\left(\sum_{i=1}^{\infty}\frac{U_{i}}{2^{i}}\right).
\end{align*}
For the last equality, we have used the properties of the cumulants of independent random variables; 
here the sum is a converging infinite series. 
Since all $U_{i}$'s have
bounded support,  $\left[-3, 3\right],$ the relation \eqref{eq:stochastic}
shows that the same holds for $Z_*.$ 
\end{proof}
We can make one step more. Namely, to use the representation \eqref{eq:stochastic} for computing 
easily at least the lower order moments of $Z_*.$ 
Below, we correctly change the order of expectation and infinite summation, because of the 
independence of $U_i$'s and the `nice' convergence properties  of 
of the infinite sum. In particular, since  ${\bf E}U_{i}^{2}=3,$ we find 
\[
{\bf E}Z_{*}^{2}={\bf E}\left(\sum_{i=1}^{\infty}\frac{U_{i}}{2^{i}}\right)^{2}=\sum_{i=1}^{\infty}\frac{{\bf E}U_{i}^{2}}{2^{2i}} 
\ \Rightarrow \ 
{\bf E}Z_{*}^{2}=3\sum_{i=1}^{\infty}\frac{1}{2^{2i}}=1.
\]
The moments of higher order can be computed similarly. Another way is to use the one-to-one relations between 
moments and  cumulants; see Shiryaev (2016). For example,   
\[
{\bf E}Z_{*}^{4}=\kappa_{4}+3\kappa_{2}^{2}=-\frac{18}{25}+3=\frac{57}{25}.
\]
More generally, we have the following result. 
\begin{thm}
The sequence of moments $m_{2n}={\bf E}Z_{*}^{2n}, \ n=1, 2, \ldots,$ satisfies
the recurrences:
\begin{equation}
m_{2n}=\frac{1}{\left(2n+1\right)\left(2^{2n}-1\right)}\sum_{j=1}^{n}\binom{2n+1}{2j+1}3^{2j}m_{2n-2j};\label{eq:rec1}
\end{equation}
\begin{equation}
m_{2n}=\sum_{j=1}^{n}\binom{2n-1}{2j-1}\frac{B_{2j}}{2j}\frac{6^{2j}}{2^{2j}-1}m_{2n-2j};\label{eq:rec2}
\end{equation}
\begin{equation}
m_{2n}=\frac{2^{2n}}{1-2^{2n}}\sum_{j=1}^{n}\binom{2n}{2j}3^{2j}B_{2j}\left(\frac{1}{2}\right)m_{2n-2j},\label{eq:rec3}
\end{equation}
where $B_{2j}\left(\frac{1}{2}\right)$ is the value of the Bernoulli
polynomial of degree $2j$ at $x=\frac{1}{2}.$ 
\end{thm}

\begin{proof}
As mentioned and used before, all odd-order moments of $Z_*$ vanish, ${\bf E}Z_{*}^{2n-1}=0$.  
Hence we work with the moments $m_{2n}={\bf E}Z_{*}^{2n}.$
Starting from the stochastic representation \eqref{eq:stochastic},
we obtain a chain of equalities: 
\begin{align*}
{\bf E}Z_{*}^{2n} & ={\bf E}\left[\sum_{i=1}^{\infty}\frac{U_{i}}{2^{i}}\right]^{2n}=\frac{1}{2^{2n}}\,{\bf E}\left[\sum_{i=1}^{\infty}\frac{U_{i}}{2^{i-1}}\right]^{2n}\\
 & =\frac{1}{2^{2n}}\,{\bf E}\left[U_{1}+\sum_{i=2}^{\infty}\frac{U_{i}}{2^{i-1}}\right]^{2n}=\frac{1}{2^{2n}}\,{\bf E}\left[U_{1}+Z_{*}\right]^{2n}\\
 & =\frac{1}{2^{2n}}\frac{1}{6}\,{\bf E}\left[\int_{-3}^{+3}\left(u+ Z_{*}\right)^{2n}\hbox{d}u\right]\\
 & =\frac{1}{2^{2n}}\,\frac{1}{6}\,\frac{1}{2n+1}\,{\bf E}\left[\left(3+ Z_{*}\right)^{2n+1}-\left(-3+Z_{*}\right)^{2n+1}\right]\\
 & =\frac{1}{2^{2n}}\,\frac{1}{6}\,\frac{1}{2n+1}\,{\bf E}\left[\left[3+ Z_{*}\right]^{2n+1}+\left[3- Z_{*}\right]^{2n+1}\right]\\
 & =\frac{1}{2^{2n}}\,\frac{1}{6}\,\frac{1}{2n+1}\,2\sum_{j=0}^{n}\binom{2n+1}{2j}3^{2n+1-2j}\,{\bf E}Z_{*}^{2j}.
\end{align*}
After simple algebra (since $m_{2n}$ appears in both the left and the right-hand
sides), we have 
\begin{equation}
\frac{m_{2n}}{3^{2n}}=\frac{1}{\left(2n+1\right)\left(2^{2n}-1\right)}\sum_{j=0}^{n-1}\binom{2n+1}{2j}\frac{m_{2j}}{3^{2j}},
\end{equation}
which is equivalent to \eqref{eq:rec1}.

The second identity, \eqref{eq:rec2},  can be derived from the classical one-to-one moments-cumulants
relations. 

The third one, \eqref{eq:rec3}, can be obtained by observing that, from \eqref{eq:stochastic}, 

\[
2Z_* \ \overset{d}{=} \ \sum_{i=1}^{\infty}\frac{U_{i}}{2^{i-1}} \ \overset{d}{=} \ U_{1}+\sum_{i=1}^{\infty}\frac{U_{i+1}}{2^{i}}, 
\]
so that the following distributional relations hold: 
\[
2Z_* \overset{d}{=} Z_*+U_{\left[-3, 3\right]} \overset{d}{=} Z_*+6U_{\left[0,1\right]}-3. 
\] 
Thus, considering the moment generating functions of both sides, we obtain  
\[
{\bf E}\hbox{e}^{2tZ_*}={\bf E}\hbox{e}^{tZ_*}{\bf E}\hbox{e}^{6tU_{\left[0,1\right]}}\hbox{e}^{-3t}.
\]
Noticing that
\[
\frac{1}{{\bf E}\hbox{e}^{tU_{\left[0,1\right]}}}=\frac{t}{\hbox{e}^{t}-1}={\bf E}\hbox{e}^{t({\bf i} L-\frac{1}{2})},
\]
where the random variable $L$ is defined by its density \eqref{sech}, we obtain that 
\[
{\bf E}\hbox{e}^{2tZ_*}{\bf E}\hbox{e}^{6t\left({\bf i}L-\tfrac{1}{2}\right)}\hbox{e}^{3t}={\bf E}\hbox{e}^{tZ_*} \ 
\Rightarrow \ Z_* \overset{d}{=} 2Z_*+6\left({\bf i}L-\frac{1}{2}\right)+3 \overset{d}{=} 2Z_*+6{\bf i}\,L.
\]
Calculate the moment of order $2n$ of each side in the last relation, and use the fact, 
by the symmetry of $L$, 
 that ${\bf E}\left({\bf i}L\right)^{2n}=B_{2n}\left(\frac{1}{2}\right).$ We obtain
\[
m_{2n}=2^{2n}{\bf E}\left(Z+3B\left(\frac{1}{2}\right)\right)^{2n}=2^{2n}\sum_{k=0}^{n}\binom{2n}{2k}m_{2n-2k}3^{2k}B_{2k}\left(\frac{1}{2}\right), 
\]
and hence 
\[
m_{2n}=\frac{2^{2n}}{1-2^{2n}}\sum_{k=1}^{n}\binom{2n}{2k}3^{2k}B_{2k}\left(\frac{1}{2}\right)m_{2n-2k}, 
\]
which is the desired result. 
\end{proof}

Among the natural further questions arising is the following one: Can we derive the asymptotic behavior
of $m_{2n}$ from the recurrences in Theorem 11? For example, as an illustration, based on \eqref{eq:rec3}, we have found approximate numerical values for the
moments $m_{2n}$ for `small' $n,$ they are given in the next table. 

\vspace{0.5cm}
\begin{center} 
\begin{tabular}{|c|c|c|c|c|c|c|c|c|}
\hline 
$n$  & $1$  & $2$  & $3$  & $4$  & $5$  & $6$  & $7$  & $8$\tabularnewline
\hline 
\hline 
$m_{2n}$  & $1$  & $\frac{57}{25}=2.28$  & $\frac{1749}{245}=7.13878$  & $26.785$  & $113.33214$  & $523.1019$  & $2580.48$  & $13420$\tabularnewline
\hline 
\end{tabular}
\par\end{center} 

\vspace{0.4cm}
\begin{rem} 
If we use relation \eqref{eq:rec3} and induction arguments, we arrive at an upper bound for the moments $m_{2n}={\bf E}Z_*^{2n}$ and 
easily conclude a convergence property for its Lyapunov quantity:  
\[
m_{2n} \leq 9^n \ \mbox{ and } \ \lim_{n \to \infty} (m_{2n})^{1/2n} = 3. 
\]

It is worth mentioning that if having only these properties of the moments of a symmetric distribution, we can conclude that 
the distribution has a support $[-3,3].$ 
\end{rem}

\section{More General Model}

It may look strange why exactly the interval $[-3,3]$ appeared as the support of the limiting distribution of the random variable $Z_*$. 
The situation becomes clear after analysing the following more general model.  

Let us use notations similar to those in the Introduction. Instead of the `base' 2,  now we involve an arbitrary positive integer $a \geq 2$ and  define the sample space as follows: 
\[
\Omega_N(a) = \{\omega\}, \ \omega=(k_1,k_2,\ldots,k_N), \quad 0\leq k_j \leq a^j - 1, \ j=1,\ldots, N 
\]
\[
\Rightarrow \ |\Omega_N(a)| = a^{N(N+1)/2}.
\]
To each outcome $\omega$ we assign the same probability: \ ${\bf P}(\omega)=1/a^{N(N+1)/2}.$    

As before, we first define $N$ independent discrete uniform random variables $X_j(a)$, $j = 1, 2, \ldots, N$, 
\[
{\bf P}[X_j(a)=k]= \frac{1}{a^j}, \ k=0, 1, \ldots, a^j - 1,  
\]
and then consider the  random variables 
\[
Z_N(a) = X_1(a) + \ldots + X_N(a), \ N=1, 2, \ldots.   
\]

Define the sequence of integer numbers \  $\{\alpha_{k}^{(N)}(a)\}$  as follows:
\[
\alpha_{k}^{\left(N\right)}(a)=|\left\{ \left(k_{1},\dots,k_{N}\right):0\le k_{j} \le a^j - 1,  j=1, \ldots, N, \  
 \Sigma_{j=1}^Nk_j=k\right\}|,
\]
where \ $k=0, 1,2,\ldots, a(a^N -1)/(a-1)-N.$  This range of $k$ comes from the moment generating function  
$\Sigma_k \alpha_k^{(N)}(a)\,x^k = \Pi_{j=1}^{N} (x^{a^j} - 1)/(x-1).$ 
The sequence $\{\alpha_{k}^{(N)}(a)\}$ is related to restricted partitions of integers. Important for us is the fact that   
\[
{\bf P}[Z_N(a)=k]=\frac{\alpha_k^{(N)}(a)}{a^{N(N+1)/2}}, \ k=0, 1,2,\ldots, a(a^N -1)/(a-1)-N
\]  
and we are interested in the standardized random variable 
\[
{\hat Z}_N(a) = \frac{Z_N(a) - \mu_N(a)}{\sigma_N(a)}.
\]
\begin{prop}
(1) The mean value  and the variance of $Z_{N}(a)$ are: 
\[
\mu_N(a)={\bf E}Z_{N}(a)=\frac{1}{2}\left[a\,\frac{a^{N}-1}{a-1}-N\right], 
\]
\[
\sigma_{N}^{2}(a)= {\bf Var}Z_N(a)=\frac{1}{12}\left[a^{2}\,\frac{1-a^{2N}}{1-a^{2}}-N\right].
\]
(2) The following limiting relation holds: 
\[
{\hat Z}_N(a) \overset{d}{\to} Z_*(a) \ \mbox{ as } N \to \infty, \ \mbox{ where } \ 
{Z}_{*}(a) \overset{d}{=}\sum_{k=1}^{\infty}\frac{U_{k}(a)}{a^{k}}.
\] 
Here $U_{k}(a), \ k=1,2, \ldots,$ is an infinite sequence of independent copies of a random variable $U(a)$ which is 
absolutely continuous and uniformly distributed
over the interval $\left[-3\sqrt{\frac{a^{2}-1}{3}}, \ 3\sqrt{\frac{a^{2}-1}{3}}\right].$ 
\end{prop}
 
\begin{rem}
It is easy to see that $a=2$ is exactly the case analysed in all details, see the previous sections. This is why we do not provide 
details here. 
\end{rem}

The statements given above can be reformulated in an equivalent form. To do this, we introduce a random variable, say $V$, 
which is continuous and uniform on the interval $[0,1]$ and let $F$ and $f$ be the easily expressible distribution function and density, 
respectively. Let further $V_1, V_2, \ldots$ be an infinite sequence of independent copies of $V$. 
The `new' quantity $V_*$, where    
\[ 
V_* = \sum_{k=1}^{\infty} \frac{V_k}{2^k}, 
\]
is a well-defined random variable with values in $[0,1]$. The distribution function $F_*$ and the density $f_*$ of $V_*$ are 
easily written as infinite convolutions, respectively of $F$ and $f$, appropriately rescaled.  

While both $F$ and $f$ have simple form, as far as we know, there are no available (short, compact, closed) 
formulas for $F_*$ and $f_*$. Eventually we may use their good approximations.   

Let us return to the random variable $Z_*.$ Since $Z_*=\sum_{k=1}^{\infty} U_k/2^k$, \  
$U_k \overset{d}{=} U$ is uniform on $[-3,3]$,  we have 
\[
U \ \overset{d}{=} \ -3 + 6\,V \quad \Rightarrow \quad Z_* \ \overset{d}{=} \ -3 + 6\,V_*. 
\]
Therefore Proposition 4.1 has the following equivalent form: 

\begin{prop}
The distribution function $G_*$ and the density $g_*$ of the random variable $Z_*$ can be expressed in terms of $F_*$ and $f_*$ as follows:
\[
G_*(x)={\bf P}[Z_* \leq x]={\bf P}[-3+6\,V_* \leq x] = F_*\left(\frac{x+3}{6}\right),
\]
\[ g_*(x) = \frac{1}{6} \,f_*\left(\frac{x+3}{6}\right), 
 x \in {\mathbb R}.
\]
\end{prop}

We treat similarly the more general random variable $Z_*(a)$, \ $a \geq 2$. Denote its distribution function and density 
by $G_*^{(a)}$ and $g_*^{(a)}$, respectively. Now $Z_*(a)=\sum_{k=1}^{\infty} U_k(a)/2^k$, where 
$U_k(a) \overset{d}{=} U(a)$ is uniform on the interval $[-b_a, b_a]$, where we have used the notation 
\[
b_a = 3{\sqrt {\frac{a^2-1}{3}}}. 
\]
We use the random variables $V$ and $V_*$ introduced above to easily write the following: 
\[
U(a) \ \overset{d}{=} \ -b_a + 2\,b_a\,V \quad \Rightarrow \quad Z_*(a) \ \overset{d}{=} \ -b_a + 2\,b_a\,V_*. 
\]
Therefore Proposition 5.1(2) has the following equivalent form: 

\begin{prop}
The distribution function $G_*^{(a)}$ and the density $g_*^{(a)}$ of the random variable $Z_*(a)$ can be expressed in terms of $F_*$ and $f_*$ as follows:
\[
G_*^{(a)}(x)={\bf P}[Z_*(a) \leq x]={\bf P}[-b_a+2\,b_a\,V_* \leq x] = F_*\left(\frac{x+b_a}{2\,b_a}\right), 
\]
\[
g_*^{(a)}(x) = \frac{1}{2\,b_a}\,f_*\left(\frac{x+b_a}{2\,b_a}\right), \ 
 x \in {\mathbb R}.
\]
\end{prop}

\section{Appendix}

Here we provide more details about the numbers $\alpha_k^{(N)}$, see the Introduction. Hence we present properties for 
$\alpha_k^{(N)}(a)$ in the case $a=2.$ Some of these properties can be extended to arbitrary integer $a>2.$ We do not 
pursue this here.  

For an arbitrary function $g:\mathbb{R\mapsto\mathbb{R}}$, consider
the multiple sums 
\[
T_{N}\left(x\right)=\sum_{k_{1}=0}^{1}\ldots\sum_{k_{N}=0}^{2^{N}-1}g\left(x+k_{1}+\cdots+k_{N}\right)
\]
over the parallelepipedic domain $\times_{i=1}^{N}\left[0,2^{i}-1\right]$.
Here are the first cases:  
\[
T_{1}\left(x\right)=\sum_{k_{1}=0}^{1}g\left(x+k_{1}\right)=g\left(x\right)+g\left(x+1\right),
\]
\[
T_{2}\left(x\right)=\sum_{k_{1}=0}^{1}\sum_{k_{2}=0}^{3}g\left(x+k_{1}+k_{2}\right)=g\left(x\right)+2g\left(x+1\right)+2g\left(x+2\right)+2g\left(x+3\right)+g\left(x+4\right),
\]
and 
\begin{align*}
T_{3}\left(x\right) & =\sum_{k_{1}=0}^{1}\sum_{k_{2}=0}^{3}\sum_{k_{3}=0}^{7}g\left(x+k_{1}+k_{2}+k_{3}\right)=g\left(x\right)+3g\left(x+1\right)+5g\left(x+2\right)\\
 & +7g\left(x+3\right)+8g\left(x+4\right)+8g\left(x+5\right)+8g\left(x+6\right)+8g\left(x+7\right)+7g\left(x+8\right)\\
 & +5g\left(x+9\right)+3g\left(x+10\right)+g\left(x+11\right).
\end{align*}

The sequence $\alpha_{k}^{\left(N\right)}$ of \ $2^{N+1}-N-1$ coefficients 
that appear in the sum 
\begin{equation}
T_{N}\left(x\right)=\sum_{k=0}^{2^{N+1}-N-2}\alpha_{k}^{\left(N\right)}g\left(x+k\right)\label{eq:alpha}
\end{equation}
are, for $N=1,2$ and $3,$ respectively, as follows:  
\[
1,1
\]
\[
1,2,2,2,1
\]
\[
1,3,5,7,8,8,8,8,7,5,3,1.
\]
Interestingly, these coefficients count the number of integer points
contained in hyperplanes that intersect a parallelepiped, as indicated
in the figures below in the case $N=2$ and $N=3.$ 

\begin{center}
\includegraphics[scale=0.2]{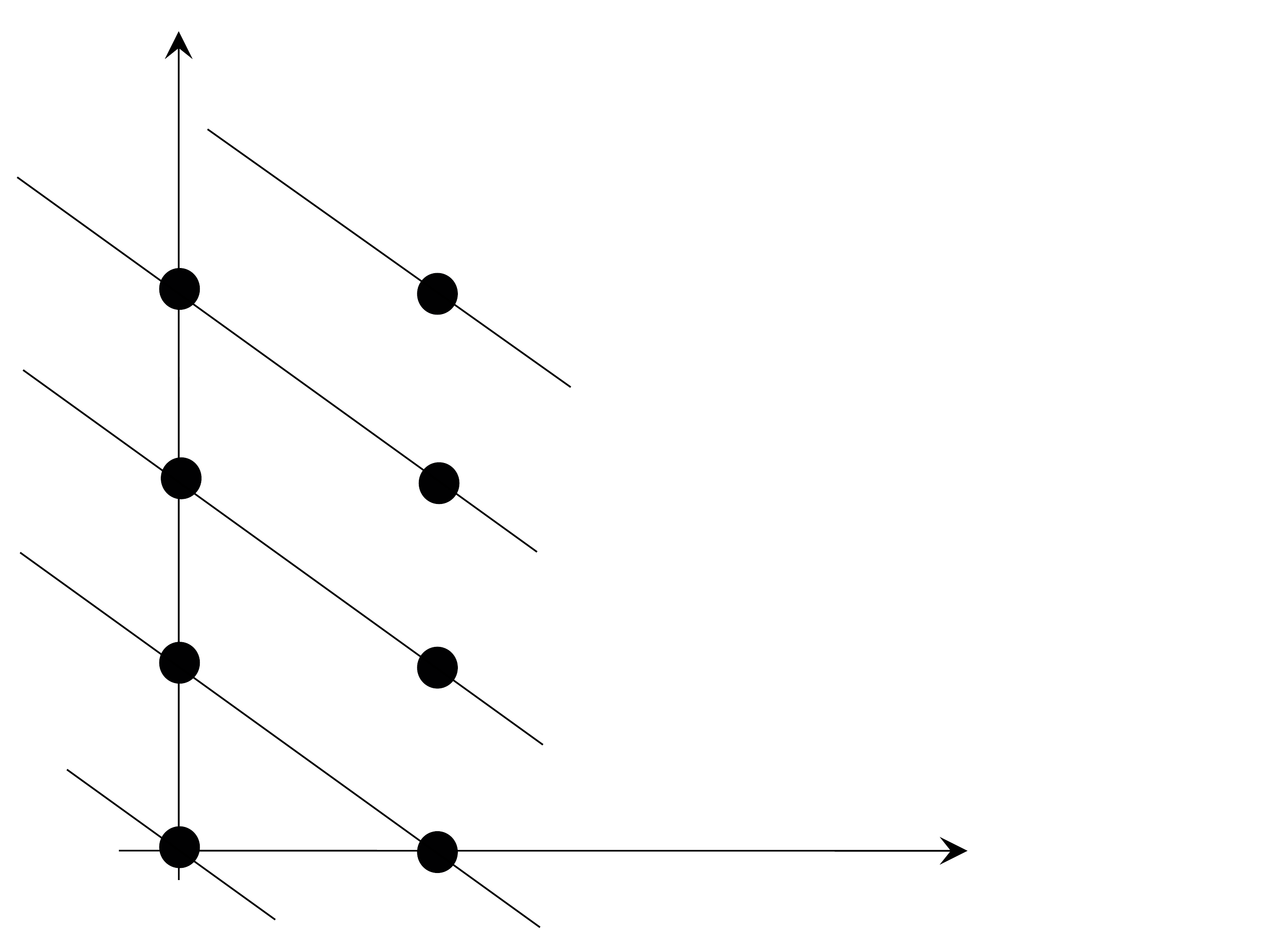} \includegraphics[scale=0.2]{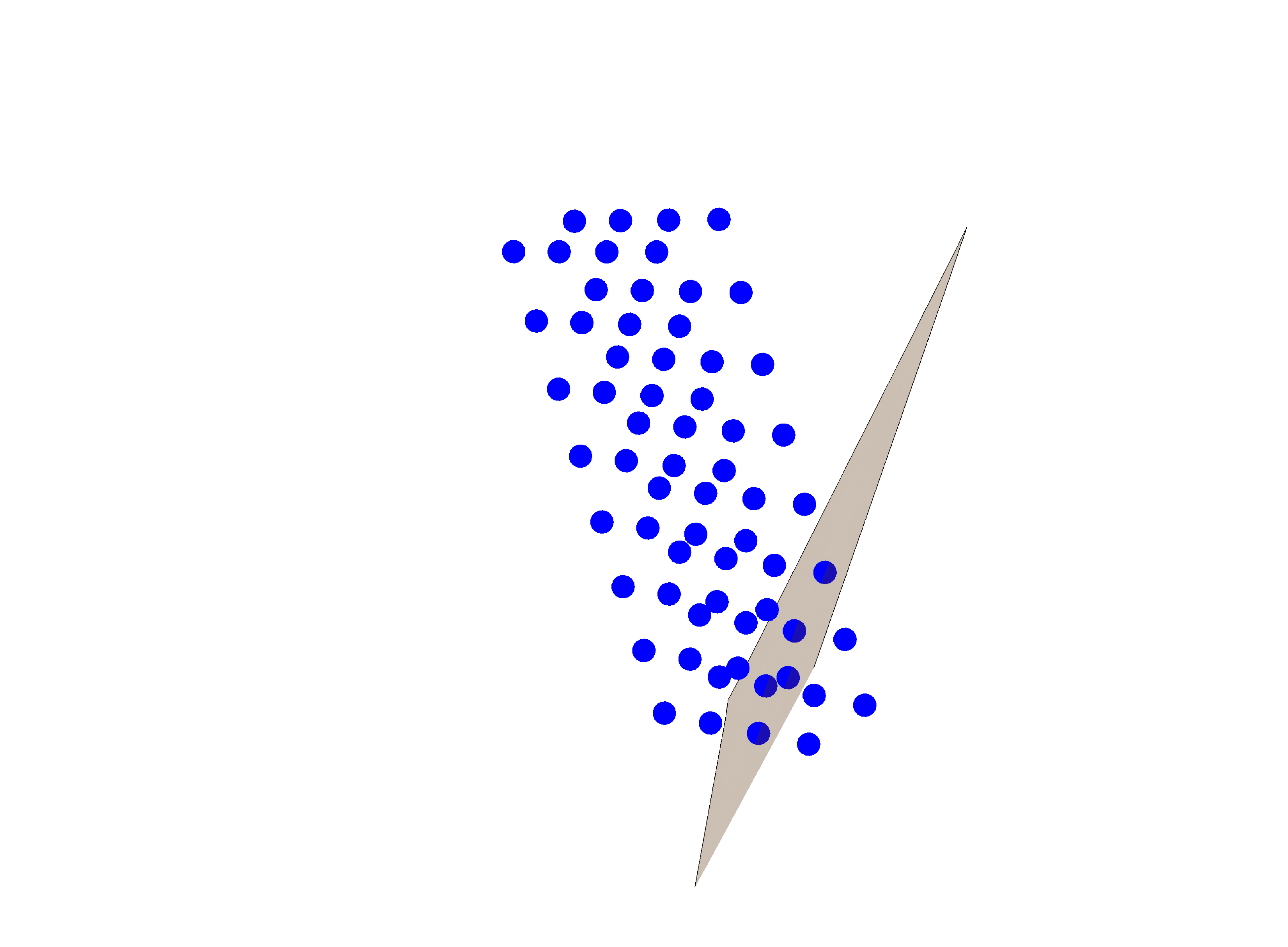} 
\par\end{center}

It is worth mentioning that the sequence $\{\alpha_{k}^{\left(N\right)}\}$
appears in the following unrelated context: Denote by $s_{2}\left(n\right)$
the number of the digits $``1"$ in the binary expansion of $n.$ Then
one can show that, for an arbitrary function $f,$ the following relation holds: 
\[
\sum_{n=0}^{2^{N}-1}\left(-1\right)^{s_{2}\left(n\right)}f\left(x+n\right)=\left(-1\right)^{N}\sum_{k=0}^{2^{N}-N-1}\alpha_{k}^{\left(N-1\right)}\Delta^{N}f\left(x+k\right).
\]
Here $\Delta$ is the finite difference operator, $\Delta f\left(x\right)=f\left(x+1\right)-f\left(x\right).$
Taking $f\left(x\right)=\hbox{e}^{tx}$ in the previous identity gives 
\[
\sum_{k=0}^{2^{N}-N-1}\alpha_{k}^{\left(N-1\right)}\hbox{e}^{kt}=\frac{1}{\left(1-\hbox{e}^{z}\right)^{N}}\sum_{n=0}^{2^{N}-1}\left(-1\right)^{s_{2}\left(n\right)}\hbox{e}^{nt}.
\]
Expanding the term $\left(1-\hbox{e}^{t}\right)^{-N}$ in an infinite series
with respect to $\hbox{e}^{t}$ and replacing $N$ by $N+1$ gives the following
expression for $\alpha_{j}^{\left(N\right)}:$ 
\begin{equation}
\alpha_{j}^{\left(N\right)}=\sum_{n=0}^{\min\{j, \ 2^{N+1}-1\}}\binom{j-n+N}{N}\left(-1\right)^{s_{2}\left(n\right)},\thinspace\thinspace\text{for\thinspace\thinspace}j = 0, 1, 2, \ldots.\label{eq:alphalN}
\end{equation}
Since there are $2^{N+1}-N-1$ coefficients $\alpha_{j}^{\left(N\right)},$
we have $0\le j\le2^{N+1}-N-2,$ hence the previous identity can
be written as 
\begin{equation}
\alpha_{j}^{\left(N\right)}=\sum_{n=0}^{j}\binom{j-n+N}{N}\left(-1\right)^{s_{2}\left(n\right)}, \ j = 0, 1, 2, \ldots, 2^{N+1}-N-2.\label{eq:alphalN2}
\end{equation}
Notice that the identity \eqref{eq:alphalN} is true for any integer $j\ge0,$
while \eqref{eq:alphalN2} is  valid only in the natural range $0\le j\le2^{N+1}-N-2.$
The expression for $\alpha_{j}^{\left(N\right)}$ in \eqref{eq:alphalN2}
can be seen as a convolution over the variable $n,$ from $n=0$ to
$n=j,$ between the sequences $\left\{ \left(-1\right)^{s_{2}\left(n\right)}\right\} $
and $\left\{ \binom{n+N}{N}\right\} .$ Moreover, it can be inverted:
starting from 
\[
\left(1-\hbox{e}^{z}\right)^{N}\sum_{k=0}^{2^{N}-N-1}\alpha_{k}^{\left(N-1\right)}\hbox{e}^{kz}=\sum_{n=0}^{2^{N}-1}\left(-1\right)^{s_{2}\left(n\right)}\hbox{e}^{nz}
\]
and expanding the left-hand side gives 
\[
\left(-1\right)^{s_{2}\left(n\right)}=\sum_{k=0}^{2^{N}-1}\binom{N}{k}\left(-1\right)^{k}\alpha_{n-k}^{\left(N-1\right)},\thinspace\thinspace0\le n\le2^{N}-1.
\]
Notice that the left-hand side depends on $N$ via the condition $0\le n\le2^{N}-1.$

\section*{Acknowledgments}

The material in this paper is partly based upon work by the second named author who was supported by the NSF (Grant No. DMS-1439786) to  
visit the Institute for Computational and Experimental Research in Mathematics in Providence (RI) during the `Point Configurations in Geometry, Physics and Computer Science' Semester Program, Spring 2018.
 
The first named author acknowledges the support provided by Academia Sinica, Institute of Statistical Science (Taiwan, RoC), for a research visit, November 2018, 
when the present paper was at the last stage of completion and when, for the first time, the results were presented publicly at a seminar talk.   

Our thanks are addressed to H.K. Hwang (Taipei),  A. Gnedin (London), L. Devroye (Montreal), P. Salminen and G. H\"{o}gn\"{a}s (Turku) for sharing with us constructive and useful comments on the topic.

\end{document}